\documentclass[10pt, a4paper]{article}

\usepackage[T1]{fontenc}
\usepackage{amsmath}
\usepackage{amsthm}
\usepackage{amsfonts}
\usepackage{bbm}

\usepackage{cite}

\usepackage{hyperref}

\theoremstyle{plain}
\newtheorem{thm}{Theorem}
\newtheorem{lem}[thm]{Lemma}

\theoremstyle{definition}

\def\P{\mathbb{P}}
\def\E{\mathbb{E}}
\def\R{\mathbb{R}}

\newcommand{\EE}[1]{\E\left[{#1}\right]}

\newcommand{\1}[1]{\mathbbm{1}_{#1}}
\newcommand{\di}{\,\textrm{d}}

\let\BFseries\bfseries\def\bfseries{\BFseries\mathversion{bold}} 

\title{Limit theorems for random walks with absorption}

\author{Micha Buck}

\begin{document} 
\maketitle

\begin{abstract}
We introduce a class of absorption mechanisms and study the behavior of real-valued centered random walks with finite variance that do not get absorbed. Our main results serve as a toolkit which allows obtaining persistence and scaling limit results for many different examples in this class. Further, our results reveal new connections between results in \cite{Kemperman1961} and \cite{Vysotsky2015}.
\end{abstract}


\section{Introduction}
\subsection{General introduction}

We denote by $S_0,S_1,\ldots$ a real-valued random walk with i.i.d. increments $X_1,X_2,\ldots$ starting in $S_0 = x \in \R$ and denote by $\P_x$ the corresponding probability measure. We assume that $\EE{X_1} = 0$ and $\sigma^2 := \mathbb{V}[X_1] \in (0,\infty)$. 
The analysis of a random walk until its first zero-crossing is a classical theme.
Often, one is first interested in the asymptotic behavior of the so-called persistence probabilities $\P_x( S_m \geq 0 \colon m \leq n )$ for $x \geq 0$, as $n \to \infty$. 
In our setup, a powerful theory for these 
persistence probabilities  is available, which goes back to Sparre Andersen and Rogozin, see e.g. \cite{Andersen1953a}, \cite{Andersen1954}, \cite{Rogozin1971}. For instance, it is well-known that under the above assumptions 
\begin{equation}
\label{eq:classicPersistence}
\P_x( S_m \geq 0 \colon m \leq n ) \sim c_x n^{-1/2}, \quad \text{as } n \to \infty,
\end{equation}
where $c_x>0$ denotes a constant depending on the distribution of $X_1$ and $x$.
Moreover, random walks conditioned on the event $\{ S_m \geq 0 \colon m \leq n \}$ have been studied deeply. Here, for limit theorems, see for instance \cite{Bolthausen1976}, \cite{Iglehart1974} and \cite{Bertoin1994}. For a recent account of the classical theory in continuous time we refer to \cite{Doney2007}.

Furthermore, modifications of the classical questions have attracted attention in the literature. For instance, in \cite{Kemperman1961}, a model is studied where the random walk can stay a geometrically distributed time below zero instead of getting immediately killed when crossing zero. In \cite{Vysotsky2015}, random walks that avoid a bounded Borel set with non-empty interior are studied. The latter problem is in turn related to 
the study of 
persistence probabilities of iterated random walks.
In continuous time, we refer to \cite{Doering2018} and \cite{Doering2018a}, where stable processes and L\' evy processes with zero mean and finite variance, respectively, are studied that avoid an interval.  
Lastly, for a broader view, we refer to two recent works: In \cite{Kabluchko2017}, classical persistence results are generalized to multidimensional random walks and in \cite{Denisov2015b} the asymptotic behavior of a multidimensional random walk in a general cone is studied.
For a recent overview on persistence, we refer to the surveys \cite{Majumdar1999}, \cite{Bray2013a}, \cite{Aurzada2015a}.

We will introduce a class of absorption mechanisms in the one-dimensional case which generalize the classical situation. 
The aim of this paper is to study the behavior of random walks that do not get absorbed.  
More precisely, we will denote by $\tau$ the time of absorption and study the asymptotic behavior of the persistence probabilities $\P(\tau > n)$, as $n \to \infty$, and prove scaling limit results for the random walk conditioned on the event $\{ \tau > n \}$, as $n \to \infty$. 

Our model brings 
the situations in \cite{Kemperman1961} and \cite{Vysotsky2015} together. 
Further, since our results turn out to be quite robust, we can provide natural generalizations of the situation in \cite{Kemperman1961} including new scaling limit results and a leeway for many more examples.

At this point we stress that in \cite{Kemperman1961} and \cite{Vysotsky2015}, respectively, completely different techniques are used: 
In \cite{Kemperman1961}, where the random walk can only stay a geometrically distributed time below zero, the approach from the classical situation is generalized. More precisely, Theorem 15.1 in \cite{Kemperman1961} includes a generalized version of Sparre Andersen's Formula, which can be used to obtain persistence results by means of Tauberian theorems as in the classical case, see e.g. Section XII.7 in \cite{Feller1971}. 
In contrast, in \cite{Vysotsky2015}, where the random walk avoids a bounded Borel set with non-empty interior, it is used that the random walk starts afresh after every jump over the Borel set that is avoided. Then, the times of jumps and the corresponding sizes of overshoots are controlled to 
use classical persistence and scaling limit results afterwards.
Moreover, we mention that \cite{Port1967} contains a further approach by means of potential theory and the results there cover the most important part of the persistence result in \cite{Vysotsky2015}.

%

Our absorption model is defined as follows: 
We denote by $T_k$ the time of the $k$-th zero-crossing, so we set $T_0 := 0$ and \[ T_{k+1} := \inf\{ n > T_k \colon S_n < 0 ,\ S_{T_k} \geq 0 \text{ or } S_n \geq 0 ,\ S_{T_k} < 0 \} .\]
Further, we let $U$ denote either a real-valued random variable or a sequence $U^{(0)},U^{(1)},\ldots$ of (not necessarily independent) real-valued random variables. Let $U_0,U_1,\ldots$ be independent copies of $U$ that are also independent of the random walk. Moreover, let $K_i \colon \R \times \R \to \{0,1\}$ or, respectively, $K_i \colon \R^{\mathbb{N}_0} \times \R \to \{0,1\}$ be measurable functions for $i \in \mathbb{N}_0$.
Then, we define   
\[ 
\tau := \inf\{ n \colon \exists k \geq 0 \text{ such that } T_k \leq n < T_{k+1} \text{ and }  K_{n-T_k}(U_k,S_n) = 1 \}. 
\]
The family of functions $(K_i)$ describes the mechanism how the random walk gets absorbed. This mechanism depends on the passed time since the last zero-crossing, some random input and the position of the random walk. 
For example, if we want to model the situation in \cite{Kemperman1961}, we choose $U$ geometrically distributed with parameter $q \in (0,1)$ and $K_i(u,x)=1$ if $x<0$, $i \geq u$ and $K_i(u,x)=0$ otherwise. 
Then, the random walk gets absorbed when it is negative and the time spent below zero exceeds an independent geometrically distributed input. 
We emphasize at this point that the situation with arbitrary distributed $U$ is also covered by our results, whereas it was crucial in \cite{Kemperman1961} that $U$ is geometrically distributed, so that the Markovian structure is preserved.
To model the situation in \cite{Vysotsky2015} is even simpler, since the random input $U$ is not needed. Let $B$ be the Borel set that is avoided by the random walk. Then, we set $K_i(u,x)=1$ if $x \in B$ and $K_i(u,x)=0$ otherwise.

The outline of this paper is as follows: 
In the next subsection, we will present our main results and give a few comments on these results.
In Section~\ref{sec:examples}, we provide several examples of absorption mechanisms and apply our theorems.
Auxiliary statements can be found in Section~\ref{sec:preliminaries}. 
After fixing notation in Subsection~\ref{sub:notation}, we collect results that do not use assumptions from the absorption model in Subsection~\ref{sub:resultsRW}.
Results that are based on the absorption model can be found in Subsection~\ref{sub:resultsModel}. Finally, we prove our main results in Section~\ref{sec:results}.

\subsection{Main results}
\label{sub:heuristic}

Let us first fix some notation, which we will use to state our results. We define
$ \E[ X ; A ] := \E[ X \1{A} ]$, 
where $\1{A}$ denotes the indicator function of the measurable set $A$.
We will fix a sequence $(a_n)$ of positive real numbers with $a_n=o(1)$ such that $a_n n^{1/2} \nearrow \infty$. 
Further, for a sequence $(A_n)$ of non-empty subsets of $\R$ and positive sequences $(f_x(n))$, $(g_x(n))$ depending on $x \in \R$, we say $f_x(n) \sim g_x(n)$ uniformly in $A_n$ if $\sup_{x \in A_n} \left| \frac{f_x(n)}{g_x(n)} - 1\right| \to 0$, as $n \to \infty$.

Our main results will reduce persistence and scaling limit problems to related problems corresponding to the stopping time $\min(T_1,\tau)$ instead of $\tau$.
These problems in turn can be reduced in many cases of interest 
to very well understood classical problems. 

Before stating the results, we will take a short look at our assumptions  
and give a few comments on them. 
We will assume that there are constants $c>0$ and $\gamma \in (0,1)$ such that, for all $x \in \R$ and $k \in \mathbb{N}_0$, 
\begin{equation}
\label{eq:assumptionModel}
\P_x( \tau \geq T_k ) \leq c \gamma^k .
\tag{C1}
\end{equation}
We can think of $\gamma$ as 
a surviving-fee that the random walk has to pay each time it crosses zero. 
This assumption is crucial for our results. 
The following two assumptions encode the relation to the simpler problem 
corresponding to the stopping time $\min(T_1,\tau)$. 
We assume that there is a function $u \colon \R \to \R$ such that
\begin{equation}
\label{eq:thmConvU}
n^{1/2} \cdot  \P_y(\tau > n , T_1 > n)  \to  u(y), \quad \text{as } n \to \infty, \text{ for all }y \in \R.
\tag{C2}
\end{equation}
Further, we only study the asymptotic behavior of a random walk in such 
an absorption model 
if the random walk starts in a point $x$ such that a constant $c$ can be chosen such that, for all $n \in \mathbb{N}$,
\begin{equation}
\label{eq:assumptionOrder0}
\P_x(\tau > n) \geq c^{-1}n^{-1/2}.
\tag{C3}
\end{equation}
The latter assumption guarantees that the probabilities of certain considered events are of the same order as the classical persistence probabilities in \eqref{eq:classicPersistence}.
In many cases of interest, these conditions can be verified relatively easily.    
For instance, let us 
again consider the situation in \cite{Kemperman1961}. 
Condition \eqref{eq:assumptionModel} clearly 
holds
for a suitable choice of $c$ and $\gamma$, since \[ \P_x(\tau \geq T_k) \leq \P(U>0)^{(k-1)/2}=(1-q)^{(k-1)/2}. \]
Further, by \eqref{eq:classicPersistence}, the left term in \eqref{eq:thmConvU} converges to $c_y$ for $y\geq 0$, since in this case $\{ \tau > n , T_1 > n \}=\{T_1 > n\}$. For $y<0$, we have \[n^{1/2} \cdot \P_y(\tau > n , T_1 > n) \leq n^{1/2} \cdot \P(U > n)  = n^{1/2} \cdot (1-q)^{n+1} \to 0. \]
Moreover, by \eqref{eq:classicPersistence}, condition \eqref{eq:assumptionOrder0} holds for $x \geq 0$. For $x < 0$, we note that there is always a positive probability that the random walk reaches the positive half-line without getting absorbed after a fixed number of steps. Then, we can use \eqref{eq:classicPersistence} again to obtain \eqref{eq:assumptionOrder0}.

In the situation of \cite{Vysotsky2015}, where the random walk avoids a bounded Borel set with non-empty interior, conditions \eqref{eq:assumptionModel} - \eqref{eq:assumptionOrder0} are fulfilled as well. Since verifying these conditions is slightly more complicated as in the preceding example, we refer to Section \ref{sec:examples} for a treatment of this situation.

Our first main result deals with the asymptotic behavior of the persistence probabilities of a random walk with absorption. 
\begin{thm}
\label{thm:persistence}
Assume that \eqref{eq:assumptionModel} and \eqref{eq:thmConvU} hold. Then, for $x$ satisfying \eqref{eq:assumptionOrder0}, we have
\begin{equation*}
\P_x( \tau > n ) \sim V(x) n^{-1/2}, \quad \text{as } n \to \infty,
\end{equation*} 
where 
\begin{equation*}
V(x) = \sum_{k=0}^\infty \E_x[ u(S_{T_k}) ; \tau \geq T_k ].
\end{equation*} 
If the uniform condition $\sup_{|y| \leq a_n n^{1/2}} |n^{1/2} \P_y(\tau > n , T_1 > n) -  u(y) | = o(1)$ is fulfilled (instead of condition \eqref{eq:thmConvU}), 
then the statement holds uniformly in
$\{ x \colon |x| \leq a_n' n^{1/2}, \P_x( \tau > n ) \geq c^{-1}(|x|+1)n^{-1/2}\}$,
where $(a_n')$ is a sequence with $a_n'=o(1)$ and $a_n' n^{1/2}\nearrow \infty$.
\end{thm}

Our second main result concerns scaling limits of random walks conditioned to not get absorbed. 
For this purpose, let $\hat{S}_n$ 
denote the continuous process on $[0,1]$ with $\hat{S}_n(m/n):=S_m/(\sigma n^{1/2})$ and which is linearly interpolated elsewhere. As before, $\sigma^2$ is the variance of $X_1$.
For an event $A$, we denote by $Law_y( \hat{S}_n \mid A )$ the probability measure on the space $(C([0,1]),\| \cdot \|_\infty)$ corresponding to the process $\hat{S}_n$ starting in $y$ and conditioned on $A$. Here, $C([0,1])$ denotes the set of continuous functions defined on $[0,1]$ and $\| \cdot \|_\infty$ denotes the supremum norm.
Further, for a continuous stochastic process $W$ on $[0,1]$, we denote by $Law( W )$ the corresponding probability measure on $(C([0,1]),\| \cdot \|_\infty)$. We will use the symbol $\Rightarrow$ to denote weak convergence of such probability measures.

We assume that there are continuous processes $W_+$ and $W_-$ on $[0,1]$ such that, for $y$ with $u(y)>0$,
\begin{align}
\label{eq:thmScalingAssumption}
Law_y( \hat{S}_n \mid \tau > n , T_1 > n  ) \Rightarrow  \begin{cases} Law( W_+ ), & y \geq 0,\\ Law( W_- ),& y < 0.  \end{cases}
\tag{C4}
\end{align}
In many cases of interest, identity \eqref{eq:thmScalingAssumption} can be easily deduced from classical results. 
For instance, let us once again consider the situation in \cite{Kemperman1961}. 
We have already seen that in this situation $u(y)>0$ if and only if $y \geq 0$. But, for $y\geq 0$, we have $\{ \tau > n , T_1 > n \} = \{ T_1 > n \}$ and \eqref{eq:thmScalingAssumption} is covered by classical results with $W_+$ being a standard Brownian meander; see e.g. \cite{Bolthausen1976}.
For the situation in \cite{Vysotsky2015}, we refer again to Section \ref{sec:examples}.

Now, we are ready to state the second main result.

\begin{thm}
\label{thm:scaling}
Assume that \eqref{eq:assumptionModel}, \eqref{eq:thmConvU} and \eqref{eq:thmScalingAssumption} hold.
Then, for $x$ satisfying \eqref{eq:assumptionOrder0}, we have
\begin{equation*}
Law_x( \hat{S}_n \mid \tau > n  ) \Rightarrow Law(\rho W_+ + (1-\rho)W_-),
\end{equation*}
where $\rho$ denotes a random variable that is independent of $W_+$ and $W_-$ 
with $\P(\rho=1)=1-\P(\rho=0)=V(x)^{-1} \sum_{k=0}^\infty \E_x[ u(S_{T_k}) ; \tau \geq T_k , S_{T_k} \geq 0 ] \in [0,1] $.
\end{thm}

The proofs of our main results are based on the following observation: A random walk that survives a long time in such an absorption model typically crosses zero only a few times at the beginning and also the magnitude of an overshoot at a zero-crossing time is typically small. 
Once this is formalized and proved, our results can be naturally deduced from \eqref{eq:thmConvU} - \eqref{eq:thmScalingAssumption}.

Roughly, two 
facts contribute to this observation.
First, due to inequality \eqref{eq:assumptionModel}, only a few zero-crossings occur. 
Second, by results from renewal theory, which characterize the sizes of overshoots of a random walk over a fixed level, as well as classical persistence results, facts about the typical zero-crossing behavior of a random walk can be obtained, see Subsection \ref{sub:resultsRW}. Combining this with the consequences of \eqref{eq:assumptionModel} yields the above observation, which is rigorously stated in Lemma \ref{lem:jumpsBeginning}.

While the proof of Theorem \ref{thm:persistence} borrows many arguments from \cite{Vysotsky2015}, our arguments to prove Theorem~\ref{thm:scaling} are completely different from the corresponding ones in \cite{Vysotsky2015}. We give a direct proof without using deeper results characterizing tightness of probability measures on function spaces, which makes the proof less technical and the arguments might be easier to adopt in other similar situations.

\section{Examples}
\label{sec:examples}

Let us begin by recalling two results from the classical setup.
First, we state a uniform version of \eqref{eq:classicPersistence} and specify the constants $c_x$. 
It holds uniformly in 
$\{ x \colon |x| \leq a_n n^{1/2}\}$ 
that 
\begin{equation}
\label{eq:exactAsympClassic}
\P_x( T_1 > n ) \sim c_x n^{-1/2},
\end{equation}
where \[c_x = \frac{\sqrt{2} | x-\E_x[ S_{T_1} ] | }{\sigma \sqrt{\pi}};\] 
see e.g. Lemma 3 in \cite{Vysotsky2015}.
Second, as already mentioned, for $x \geq 0$, we have
\begin{equation}
\label{eq:scalingLimitClassic}
Law_x(\hat{S}_n \mid T_1 > n) \Rightarrow  Law(B_+),
\end{equation}
where $B_+$ is a standard Brownian meander. 
A proof of this statement for $x=0$ can be found in \cite{Bolthausen1976}.
The result for $x>0$ follows, for example, with the techniques in our proof of Theorem \ref{thm:scaling}. In the following, we will denote by $B_+$ a standard Brownian meander and set $B_-:=-B_+$.

Now, we will give several examples of absorption mechanisms. 
In particular, we will show how the situation in \cite{Kemperman1961} and \cite{Vysotsky2015} are covered by our results, and thus, how these situations are connected.
Since our results can be easily applied in the following examples, we just sketch 
how conditions \eqref{eq:assumptionModel} - \eqref{eq:thmScalingAssumption} can be verified and leave the details to the reader.

First, let us consider two natural generalizations of the situation in \cite{Kemperman1961}, where the random walk is allowed to stay a geometrically distributed time $U$ below zero. The assumption of $U$ being geometrically distributed is crucial in \cite{Kemperman1961} because the techniques heavily rely on the Markovian and the homogeneous structure of the process. Our approach does not need this assumption and covers also more sophisticated situations. 
Further, no scaling limit results have been proved in the literature.
\begin{itemize}
\item 
\emph{Random times below zero:}\\ 
Let $U$ be a non-negative random variable. We allow that $\P(U=\infty)>0$. We set $K_i(u,x)=1$ if $i \geq u$, $ x < 0$ and $K_i(u,x)=0$ otherwise.
We can think of this model as follows:
Every time the random walk enters the negative half-line, it can only survive an independent random time, according to the distribution of $U$, below zero.

The case $\P(U=\infty)=1$ is trivial. 
Thus, let us exclude this case. 
Now, we choose $u_0$ such that $\P(U>u_0)<1$. Then, \eqref{eq:assumptionModel} is fulfilled, since
\[
\P_x( \tau \geq T_k ) \leq \P_x( U > u_0 \text{ or } T_1 \leq u_0 )^{(k-1)/2} \leq \P(U>u_0)^{(k-1)/2}.
\]
For $y \geq 0$, we note that $\{ \tau > n , T_1 > n \} = \{ T_1 > n \}$, and thus, by \eqref{eq:exactAsympClassic}, \eqref{eq:thmConvU} holds with $u(y)=c_y$.
Further, again by \eqref{eq:exactAsympClassic}, we obtain, for $y<0$, that 
\[ n^{1/2} \cdot \P_y(\tau > n , T_1 > n) = n^{1/2} \cdot \P_y(T_1>n)  \P(U>n) \to c_y \P(U=\infty), \]
and thus, \eqref{eq:thmConvU} holds with $u(y)=c_y \P(U=\infty) $.
If $U$ and $X_1$ are bounded from above, clearly \eqref{eq:assumptionOrder0} does not hold for $x$ below a certain negative level. To verify \eqref{eq:assumptionOrder0} for all other $x$, we force the random walk to start with a certain number of positive jumps, so that it reaches the non-negative half-line with a positive probability and use \eqref{eq:exactAsympClassic} afterwards.
Now, we can apply Theorem~\ref{thm:persistence}. 
By the same argument, we also obtain the uniform statement in Theorem~\ref{thm:persistence}. 
Finally, by \eqref{eq:scalingLimitClassic} and an analogous argument as above, we obtain, for $y$ with $u(y)>0$,
\begin{align*}
Law_y( \hat{S}_n \mid \tau > n , T_1 > n ) \Rightarrow  \begin{cases} Law(B_+), &  y \geq 0, \\ Law(B_-), &  y < 0 . \end{cases}
\end{align*}
Hence, we can apply Theorem \ref{thm:scaling}.
\item 
\emph{Inhomogeneous absorption probabilities:}\\ 
We start with a measurable function $p \colon \R \to [0,1]$  
with $p(x)=0$ for $x\geq 0$ and $\liminf_{x \to -\infty} p(x) > 0$.
Further, we let $U$ be a sequence $U^{(0)},U^{(1)},\ldots$ of independent random variables which are uniformly distributed on $[0,1]$. Let $u$ denote a sequence $u^{(0)},u^{(1)},\ldots$ of real numbers in $[0,1]$.
Then, we set $K_i(u,x)=1$ if $p(x) \geq u^{(i)}$ and $K_i(u,x)=0$ otherwise.
We can think of $p(x)$ as the probability of an absorption at the point $x$. Thus, in every step the process gets absorbed with a probability according to its current position. 

For $y<0$, we consider the event where the random walk starts with a fixed number $n_0$ of negative jumps, so that the probability $p_0$ of an absorption until time $n_0$ is positive. 
Thus, it follows that \eqref{eq:assumptionModel} holds, since $\P_x(\tau \geq T_k) \leq (1-p_0)^{(k-1)/2}$.
Using the same idea, we obtain that, for $y<0$,
\[
n^{1/2} \cdot \P_y( \tau > n , T_1 > n ) \leq n^{1/2} \cdot (1-p_0)^{\lfloor n/n_0 \rfloor} \to 0.
\]
For $y \geq 0$, we note that $\{ \tau > n , T_1 > n \} = \{ T_1 > n \}$. Thus, it follows directly from \eqref{eq:exactAsympClassic} that 
\eqref{eq:thmConvU} holds with $u(y)=c_y$. Likewise, it follows directly from \eqref{eq:scalingLimitClassic} that, for $y \geq 0$,
\begin{align*}
Law_y( \hat{S}_n \mid \tau > n , T_1 > n ) \Rightarrow Law(B_+).
\end{align*}
Now, we can apply Theorem \ref{thm:persistence} and Theorem \ref{thm:scaling}. 
Note that the uniform statement in Theorem \ref{thm:persistence} can be obtained by the same argument.
Trivially, by considering $(-S_n)$, one obtains results for the case where $\liminf_{x \to \infty} p(x) > 0$ and $p$ vanishes below a given level.
\end{itemize}
Let us now turn our attention to the situation in \cite{Vysotsky2015}, where random walks that avoid a bounded Borel set are considered. In the next example, we will see how this situation is covered by our results. Further, we make the situation in \cite{Vysotsky2015} more complex by allowing some additional randomness to demonstrate how robust our toolkit is. Moreover, we consider a converse situation 
where the random walk is forced to pass an interval when crossing zero.
\begin{itemize}
\item
\emph{Avoiding random sets:}\\
Let $U$ be a discrete random variable on $\mathbb{N}_0$ and let $B_0,B_1,\ldots$ be a sequence of bounded Borel sets in $(-\infty,0)$.
Further, 
let us assume that $B_0$ has a non-empty interior, 
$\P(U=0)>0$ and $(S_n)$ is non-arithmetic, that is $\P(X_1 \in d\mathbb{Z})<1$ for all $d > 0$. 
We set $K_i(u,x)=1$ if $x \in B_u$ and $K_i(u,x)=0$ otherwise.
Thus, the walk gets absorbed when it hits these randomly chosen sets. 
If we choose $U=0$,  
we are in the situation of \cite{Vysotsky2015}.

Using 
standard results from renewal theory, one can show that \eqref{eq:assumptionModel} holds, see Section 2 in \cite{Vysotsky2015}. 
Since the random walk cannot be absorbed in the non-negative half-line, 
for $y\geq 0$, assumptions \eqref{eq:thmConvU} and \eqref{eq:thmScalingAssumption} are covered by \eqref{eq:exactAsympClassic} and \eqref{eq:scalingLimitClassic}, respectively.
For $y<0$, let us first fix an index $u$. Then, an application of 
our last example (Inhomogeneous absorption probabilities)
with 
\begin{align*}
p(x) = \begin{cases} 1, &  x \geq 0 \text{ or } x \in B_u, \\ 0, & \text{otherwise,} \end{cases}
\end{align*}
yields to 
persistence and scaling limit results for the case that the random walk avoids the set $B_u$ and the non-negative half-line. 
In particular, we obtain for every $y<0$ a constant $c_y^{(u)}$ such that \[ n^{1/2} \cdot \P_y(  S_m \not\in B_u \text{ for } m \leq n , T_1 > n ) \to c_y^{(u)}. \]
Therefore, we obtain that \eqref{eq:thmConvU} holds with
\begin{align*}
n^{1/2}& \cdot \P_y( \tau > n , T_1 > n )\\
&= 
\sum_{u \in \mathbb{N}_0}  n^{1/2} \cdot \P_y( S_m \not\in B_u \text{ for } m \leq n , T_1 > n ) \P (U=u)\\
&\to
\sum_{u \in \mathbb{N}_0}  c_y^{(u)} \P (U=u).
\end{align*}
The convergence in the last step follows by \eqref{eq:exactAsympClassic} and the dominated convergence theorem.
Hence, we can apply Theorem \ref{thm:persistence} and Theorem \ref{thm:scaling}.
Note that the uniform statement in Theorem \ref{thm:persistence} can be obtained by the same argument.
The $d$-arithmetic case (there is a largest constant $d>0$ with $\P(X_1 \in d\mathbb{Z})=1$) can be treated analogously.
\item
\emph{Crossing zero through an interval:}\\
Let $I$ be an open interval containing zero and let $(S_n)$ be non-arithmetic. We set $K_i(u,x)=1$ if $x \not\in I$, $i=0$ and $K_i(u,x)=0$ otherwise.
In this example, the random walk is forced to hit the interval $I$ at zero-crossing times 
and
we have no random input.

Using the ideas from the last example (Avoiding random sets), the application of our theorems is straightforward.
The $d$-arithmetic case can be treated analogously.
\end{itemize}

\section{Auxiliary results}
\label{sec:preliminaries}

\subsection{Notation}
\label{sub:notation}

Let us first fix some notation. For the most part, we will follow the notation in \cite{Vysotsky2015}. 
We set
\[ H_k := S_{T_{k}}.\] 
Sometimes it is more convenient to work with one probability measure $\P$ instead of the family of probability measures $(\P_x)$. For this reason, we set $\P=\P_0$ and denote by $T_k(x)$ the $k$-th time where the random walk crosses the level $-x$.
More precisely, we set $T_0(x) := 0$ and \[ T_{k+1}(x) := \inf\{ n > T_k(x) \colon S_n < -x ,\ S_{T_k(x)} \geq -x \text{ or } S_n \geq -x ,\ S_{T_k(x)} < -x \} .\]
Accordingly, we set $H_k(x):=S_{T_k(x)}+x$.
Further, in some situations, we will use the notation
\[
p_y^{(T)}(n) := \P_y( T_1 > n ), \quad p_y^{(\tau)}(n) := \P_y( \tau > n )
\]
and
\[
p_y^{(T,\tau)}(n) := \P_y( T_1 > n , \tau > n ).
\]
As before, we
denote by $(a_n)$ a sequence of positive real numbers with $a_n=o(1)$ and $a_n n^{1/2} \nearrow \infty$, as $n \to \infty$. Now, we let $(b_n)$ be a sequence of positive integers with $a_n^2 n = o(b_n)$ and $b_n=o(n)$. In particular, it follows that $b_n \to \infty$, as $n \to \infty$. 
To avoid technical problems, let us choose the sequence $(b_n)$ such that the sequences $(b_n)$ and $(n-b_n)$ are monotonically increasing.

\subsection{Auxiliary results for random walks}
\label{sub:resultsRW}

We start by collecting some basic facts about random walks with finite variance, which will be used at several points in this article.
Let us recall that there is a constant $c>0$ such that
\begin{equation}
\label{eq:eppel0}
\P_x(T_1=n) \leq c (|x|+1)n^{-3/2}, \quad \text{for all } x \in \R,
\end{equation}
see Lemma 5 in \cite{Eppel1979}, 
which implies that $c$ can be chosen such that also
\begin{equation}
\label{eq:eppel}
\P_x(T_1>n) \leq c (|x|+1)n^{-1/2}, \quad \text{for all } x \in \R.
\end{equation}
For ease of notation, in this article, $c$ will denote a varying positive constant which can change from line to line.
It is well-known that the ladder heights of the random walk $(S_n)$ are integrable if and only if $\mathbb{V}[X_1]<\infty$. Considering the random walk with increments distributed as the positive ladder heights of the radom walk $(S_n)$, Theorem~III.10.2~(iii) in \cite{Gut2009} states that
$\E_{x}[ H_1 ] = o(|x|)$, as $x\to-\infty$. Combining this with the corresponding result for the case with increments distributed as the negative ladder heights of $(S_n)$, one obtains 
\begin{equation}
\label{eq:gut10.2.iii}
\E_{x}[ |H_1| ] = o(|x|), \quad \text{as } |x| \to \infty.
\end{equation}
Using this fact, it can be obtained that,  
for any $\alpha \in (0,1)$, there is a constant $K_\alpha$ such that, for all $x \in \R$, \[\E_x[ |H_1| ] \leq \alpha |x| + K_\alpha .
\]
Now, the strong Markov property, an iteration
procedure and the preceding estimate give 
\begin{equation}
\label{eq:boundExpectH1}
\E_x[ | H_k | ] \leq \alpha^k \vert x \vert + \sum_{j=0}^{k-1} \alpha^j K_\alpha  \leq |x| + K,
\end{equation} 
where $K := {K_\alpha}/{(1-\alpha)}$.

\begin{lem}
\label{lem:estimates_an_bn}
Let $k \in \mathbb{N}_0$ be fixed. 
Then, as $n \to \infty$, one has 
\begin{enumerate}
\item[(a1)]
$\sup_{|x| \leq a_n n^{1/2}} \P_x( |H_k| > a_n n^{1/2} ) = o(1)$, 
\item[(a2)] $\sup_{|x| \leq a_n n^{1/2}} \E_x[ |H_k| ; |H_k| > a_n n^{1/2} ]/(|x|+1)  = o(1)$, 
\item[(b1)] $\sup_{|x| \leq a_n n^{1/2}} \P_x( |T_k| > b_n ) = o(1)$, 
\item[(b2)] $\sup_{|x| \leq a_n n^{1/2}} \E_x[ |H_k| ; |T_k| > b_n ]/(|x|+1) = o(1)$. 
\end{enumerate}
\end{lem}

\begin{proof}
The main ingredients of our proof are identities \eqref{eq:eppel}, \eqref{eq:gut10.2.iii}, \eqref{eq:boundExpectH1} and the fact that the family of random variables 
\begin{equation}
\label{eq:gut10.2.ii}
\left\{ \frac{|H_1(x)|}{|x|} \colon |x| \geq 1 \right\}
\end{equation}
is uniformly integrable, which is Theorem~III.10.2~(ii) in \cite{Gut2009}. As in \eqref{eq:gut10.2.iii}, one considers the random walk with increments distributed as the positive and negative ladder heights of $(S_n)$, respectively. 
We will prove the different statements in this lemma by induction. Further, note that all statements are trivial for the case $k=0$.
\begin{enumerate}
\item[(a1)]
First, we will consider the case $k=1$. 
Note that, by Markov's inequality, 
\[ \sup_{|x| \leq a_n n^{1/2}} \P_x(|H_1|>a_n n^{1/2}) \leq \sup_{|x| \leq a_n n^{1/2}}\frac{ \E_x[ |H_1| ]}{a_n n^{1/2}} \to 0 ,\] which follows straightforwardly from \eqref{eq:gut10.2.iii}.

We will use this fact to prove (a2).
Then, the 
statement for $k \geq 2$ follows from (a2) by Markov's inequality, since
\begin{align*}
\sup_{|x| \leq a_n n^{1/2}} \P_x( |H_k| > a_n n^{1/2} ) 
&= 
\sup_{|x| \leq a_n n^{1/2}} \P_x( |H_k| \1{\{|H_k| > a_n n^{1/2}\}} > a_n n^{1/2} ) \\
&\leq
\sup_{|x| \leq a_n n^{1/2}} \frac{\E_x[ |H_k| ; |H_k| > a_n n^{1/2} ]}{a_n n^{1/2}} = o(1).
\end{align*}
\item[(a2)]
We begin by proving the statement for the case $k=1$. First, note that, for $0 \leq x < 1$, one has 
\begin{equation}
\label{eq:a2_0}
|H_1( x )| \leq |H_1(  1 )|+1 \quad \text{a.s.}
\end{equation}
because the overshoot can be estimated either by $|H_1( 1 )|+1$ if $|H_1( x )|>1$ or $1$ if $|H_1( x )| \leq 1$. 
Thus, for $0 \leq x < 1$, we have
\begin{align}
\label{eq:a2_1}
\begin{split}
\sup_{0 \leq x < 1} & \frac{\E_x[ |H_1| ; |H_1|>a_n n^{1/2} ] }{x+1} \\
& \qquad \qquad \leq
\sup_{0 \leq x < 1} \E[ |H_1(x)| ; |H_1(x)|>a_n n^{1/2} ]  \\
& \qquad \qquad \leq
\E[ |H_1(1)|+1 ; |H_1(1)|+1>a_n n^{1/2} ] \\ 
& \qquad \qquad = 
o(1).
\end{split}
\end{align}
Further, using (a1) 
for the case $k=1$ and using the uniform integrability of \eqref{eq:gut10.2.ii}, we can conclude that 
\begin{align}
\label{eq:a2_2}
\begin{split}
&\sup_{1 \leq x \leq a_n n^{1/2}} \frac{ \E_x[ |H_1| ; |H_1|>a_n n^{1/2} ]}{x+1}\\  
&\leq
\sup_{1 \leq x \leq a_n n^{1/2}} \E\left[ \frac{|H_1(x)|}{x} ; H_1(x)>a_n n^{1/2} \right] = o(1).
\end{split}
\end{align}
For negative $x$ we proceed analogously. Altogether, this shows the claim for $k=1$.

Now, we proceed by induction. Let us assume that the statement holds for $1 \leq j \leq k$, and therefore also (a1) holds for $1 \leq j \leq k$. 
We set \[e^{(j)}(y,a,b) := \E_y[ |H_j| ; |H_j| > a , T_j > b ]\] 
and, for $1 \leq j \leq k$, by the induction hypothesis, we can choose $c^{(j)}_n = o(1)$ such that 
\begin{equation}
\label{eq:def_c_n}
e^{(j)}(y,a_n n^{1/2},0) \leq c^{(j)}_n (|y|+1) \quad \text{and} \quad  \P_y( |H_j| > a_n n^{1/2} ) \leq c^{(j)}_n,
\end{equation} 
for all $ |y| \leq a_n n^{1/2}$.
Then, using the strong Markov property and splitting the event of the first jump, 
we obtain
\begin{align*}
\E_x[ |H_{k+1}| ;& |H_{k+1}| > a_n n^{1/2} ]\\
 = &
\int e^{(k)}(y,a_n n^{1/2},0) \P_x( H_1 \in \di y  )\\
 = &
\int e^{(k)}(y,a_n n^{1/2},0) \P_x( H_1 \in \di y , |H_1| > a_n n^{1/2} ) \\
&+
\int e^{(k)}(y,a_n n^{1/2},0) \P_x( H_1 \in \di y , |H_1| \leq a_n n^{1/2} ) \\
\leq &
\int (|y|+K) \P_x( H_1 \in \di y , |H_1| > a_n n^{1/2} ) \\
&+
\int c^{(k)}_n( |y|+1 ) \P_x( H_1 \in \di y , |H_1| \leq a_n n^{1/2} ) \\
\leq &
c_n^{(1)}(|x|+K+1)  +
c_n^{(k)}( |x|+K + 1),
\end{align*}
where we used \eqref{eq:boundExpectH1} 
and 
\eqref{eq:def_c_n} for the cases $j=1$ and $j=k$. 
This completes the proof of the statement.
\item[(b1)]
By \eqref{eq:eppel}, 
we have  
\begin{align*}
\sup_{|x| \leq a_n n^{1/2}} \P_x( T_1 > b_n ) &\leq \sup_{|x| \leq a_n n^{1/2}} c( |x| + 1 ) \lfloor b_n \rfloor^{-1/2}\\& \leq c( |a_n n^{1/2}| + 1 ) \lfloor b_n \rfloor^{-1/2} = o(1) ,
\end{align*}
since $a_n n^{1/2} = o(b_n^{1/2})$ and $b_n \to \infty$.
Thus, for $k=1$, the statement holds.

Again, we proceed by induction and assume that the statement holds for $1 \leq j \leq k$.
We have 
\begin{align*}
\P_x(T_{k+1}>b_n) \leq \P_x(T_k > b_n/2) + \P_x(T_{k+1}-T_k > b_n/2).
\end{align*}
The first term vanishes uniformly in $\{ x \colon |x| \leq a_n n^{1/2} \}$ due to the induction hypothesis. 
For the second term, we obtain
\begin{align*}
\P_x(T_{k+1}-T_k > b_n/2)
=&
\int p^{(T)}_y(b_n/2) \P_x( H_k \in \di y  ) \\
=&
\int p^{(T)}_y(b_n/2) \P_x( H_k \in \di y , |H_k| > a_n n^{1/2} ) \\
&+
\int p^{(T)}_y(b_n/2) \P_x( H_k \in \di y , |H_k| \leq a_n n^{1/2} )\\
\leq&
\P_x(  |H_k| > a_n n^{1/2} )
+
\sup_{|x| \leq a_n n^{1/2}} \P_x( T_1 > b_n/2 ).
\end{align*}
Now, the claim follows from statement (a1) and the induction hypothesis for the case $j=1$. 
\item [(b2)]
We proceed similarly as in the proof of (a2). By the monotonicity of $T_1(\ \cdot\ )$ and by (b1) for the case $k=1$, we obtain by the same arguments used to show \eqref{eq:a2_1} and \eqref{eq:a2_2} that  
\[ 
\sup_{|x| \leq a_n n^{1/2}} \frac{\E_x[ |H_1| ; T_1 > b_n ]}{|x|+1} = o(1),
\]
which is the statement for the case $k=1$.

Let us assume that the statement holds for $1 \leq j \leq k$. Then, by the induction hypothesis we can choose $c^{(j)}_n = o(1)$, for $1 \leq j \leq k$, such that
\begin{equation}
\label{eq:def_c_n2}
e^{(j)}(y,0,b_n/2) \leq c^{(j)}_n (|y|+1) \quad \text{and} \quad  \P_y( |T_j| > b_n/2 ) \leq c^{(j)}_n
\end{equation}
and \eqref{eq:def_c_n} hold, for all $ |y| \leq a_n n^{1/2}$.
Now, 
we obtain
\begin{align*}
\E_x[ |H_{k+1}| ; T_{k+1} > b_n ]
\leq &
\int e^{(k)}(y,0,0) \P_x( H_1 \in \di y , T_1 > b_n/2 ) \\
&+
\int e^{(k)}(y,0,b_n/2) \P_x( H_1 \in \di y , |H_1| > a_n n^{1/2} ) \\
&+
\int e^{(k)}(y,0,b_n/2) \P_x( H_1 \in \di y , |H_1| \leq a_n n^{1/2} ) \\
\leq &
\int (|y|+K) \P_x( H_1 \in \di y , T_1 > b_n/2 ) \\
&+
\int (|y|+K) \P_x( H_1 \in \di y , |H_1| > a_n n^{1/2} ) \\
&+
\int c_n^{(k)} (|y|+1) \P_x( H_1 \in \di y , |H_1| \leq a_n n^{1/2} ) \\
\leq &
c_n^{(1)}(|x|+K+1)
+
c_n^{(1)}(|x|+K+1)\\
&+
c_n^{(k)}(|x|+K+1).
\end{align*}
Here, we used \eqref{eq:boundExpectH1} and \eqref{eq:def_c_n2} for $j=k$ in the second step.
Further, in the last step, we used \eqref{eq:def_c_n} and \eqref{eq:def_c_n2} 
for the case $j=1$ and again estimate \eqref{eq:boundExpectH1}.
Now, the statement follows.
\end{enumerate}
\end{proof}

\subsection{Auxiliary results for random walks with absorption}
\label{sub:resultsModel}

\begin{lem}
\label{lem:assumptionModelConsequence}
Assume \eqref{eq:assumptionModel} holds. 
Then, $c>0$ and $\gamma \in (0,1)$ can be chosen such that, for all $x \in \R$ and $k \in \mathbb{N}_0$, \[\E_x[ |H_k| ; \tau \geq T_k] \leq c \gamma^k (|x|+1) . \]
\end{lem}

\begin{proof}
We begin by showing that, for all $\delta>0$, a constant $K \geq 1$ can be chosen such that 
\begin{equation}
\label{eq:assumptionModelConsequence_step1}
\P_x(|H_1|>K) \leq \delta, \quad \text{for all } x \in \R.
\end{equation}
Let us first assume 
that the random walk $(S_n)$ is non-arithmetic. Let us further recall, from renewal theory, that in this case 
\begin{equation}
\label{eq:probHitInterval}
\lim_{x \to \infty} \P_x( |H_1| < K ) = \frac{1}{\E_0[ | H_1 | ]} \int_0^{K} \P_0( |H_1| > t ) \di t ,
\end{equation}
see e.g. Theorem III.10.3 (i) in \cite{Gut2009}. 
It follows directly from \eqref{eq:probHitInterval} that, for $\delta>0$, 
a constant $x_0 \geq 0$ and $K$ can be chosen such that $\P_x(|H_1|<K) \geq 1-\delta$, for  $x \geq x_0$.
For $0 \leq x \leq x_0$, by the same argument that we used to show \eqref{eq:a2_0}, we obtain that \[|H_1(x)| \leq |H_1(x_0)|+x_0.\]
Thus, $K$ can be chosen such that $\P_x(|H_1|<K) \geq 1-\delta$ holds, for all $x \geq 0$. Analogously, we argue for negative $x$. 
In the $d$-arithmetic case, we use the identity 
\[
\lim_{n \to \infty} \P_{nd}( |H_1| < kd ) = \frac{d}{\E_0[ | H_1 | ]}\sum_{j=0}^{k-1}   \P_0( |H_1| > jd ) ,
\] see e.g. Theorem III.10.3 (ii) in \cite{Gut2009}, and proceed likewise.
Hence, in both cases we obtain \eqref{eq:assumptionModelConsequence_step1}. 
Moreover, we choose $K \geq 1$ to make our next arguments a little bit smoother.

We will now estimate 
the quantity $\E_x[ |H_1| ; |H_1| > K ]$. 
First, we consider 
the case $|x|\geq 1$.
Here,
since $\delta>0$ was arbitrary in \eqref{eq:assumptionModelConsequence_step1}, we obtain, due to the uniform integrability of \eqref{eq:gut10.2.ii}, that, if $K$ is chosen large enough, 
\begin{equation}
\label{eq:proofAssumption1}
\E_x[ |H_1| ; |H_1| > K ]
=
|x| \E\left[ \frac{H_1(x)}{|x|} ; |H_1| > K \right]
\leq \gamma |x| , \quad \text{for } |x| \geq 1,
\end{equation}
with $\gamma$ from \eqref{eq:assumptionModel}. 
Now, by an iteration procedure, we obtain that, 
\begin{equation}
\label{eq:proofAssumption3}
\E_x[ |H_k| ; |H_j| > K \text{ for } j \leq k ] \leq \gamma^k |x|, \quad \text{for } |x| \geq 1.
\end{equation} 
Further,
we obtain by the same argument as in \eqref{eq:a2_0} and \eqref{eq:a2_1} that
\begin{equation*}
\label{eq:proofAssumption2}
\E_x[ |H_1| ; |H_1| > K ] 
\leq 
\gamma, \quad \text{for } |x| < 1,
\end{equation*}
if $K$ is chosen large enough. 
Combing this with \eqref{eq:proofAssumption1}, we obtain 
\begin{equation}
\label{eq:proofAssumption4}
\E_x[ |H_1| ; |H_1|>K ] \leq  \gamma K, \quad \text{for } |x| \leq K.
\end{equation}

Now, we are ready to prove the statement of the lemma. For this purpose, we note that
\begin{align}
\label{eq:decompsotionEvents}
\begin{split}
\left\{ \tau \geq T_k  \right\}
\subseteq &
\bigcup_{j=0}^{k-1} \left\{ |H_j| \leq K , |H_{j+1}| > K, \ldots , |H_{k}| > K , \tau \geq T_j  \right\} \\
& \cup
\left\{ |H_0| > K , \ldots , |H_k| > K  \right\}
\cup
\left\{ |H_k| \leq K , \tau > T_k  \right\}.
\end{split}
\end{align}
Further, we obtain, for fixed $0 \leq j \leq k-1$,
\begin{align*}
& \E_x[ |H_k| ;  |H_j| \leq K , |H_{j+1}| > K, \ldots , |H_{k}| > K, \tau \geq T_j ] \\
&\leq
\int \gamma^{k-j-1} |y| \P_x( H_{j+1} \in \di y , |H_{j+1}| > K , |H_j| \leq K , \tau \geq T_j ) \\
&\leq
\int \gamma^{k-j} K \P_x( H_j \in \di y , \tau \geq T_j ) \\
&\leq
cK\gamma^k,
\end{align*}
where we used \eqref{eq:proofAssumption3} in the first step, \eqref{eq:proofAssumption4} in the second step and \eqref{eq:assumptionModel} in the last step.
Hence, by \eqref{eq:decompsotionEvents}, summing over $j$ and using \eqref{eq:proofAssumption3} and \eqref{eq:assumptionModel} yield
\[
\E_x[ |H_k| ; \tau \geq T_k ] 
\leq 
c K k \gamma^k + \gamma^k |x|  +  cK\gamma^k, \quad \text{for all } x \in \R.
\]
The claim now follows, when we choose (new) suitable constants $c > 0$ and $\gamma \in (0,1)$.
\end{proof}

Next, we will prove an a priori estimate for the persistence probabilities $\P_x(\tau > n)$ that is of the same type as the classical result in \eqref{eq:eppel}.
\begin{lem}
\label{lem:apriori}
If \eqref{eq:assumptionModel} holds, then
there is a constant $c>0$ such that 
\begin{equation*}
\P_x(\tau > n) \leq c(|x|+1)n^{-1/2}, \quad \text{for all } x \in \R.
\end{equation*}
\end{lem}

\begin{proof}
We first decompose the persistence event as
\begin{equation}
\label{eq:decomp}
\P_x(\tau > n)
=
\P_x(\tau > n, T_k > n) + \P_x(\tau > n, T_k \leq n),
\end{equation}
where $ - \ln(n)/\ln(\gamma) \leq k < - \ln(n)/\ln(\gamma) + 1$ with $\gamma$ from \eqref{eq:assumptionModel}. Then, by \eqref{eq:assumptionModel}, the second term can be estimated by
\[
 \P_x(\tau >  T_k ) \leq c \gamma^k = c e^{k \ln(\gamma)} \leq c e^{-\frac{\ln(n)}{\ln(\gamma)}\ln(\gamma)} = cn^{-1} = o(n^{-1/2})
\]
and thus is negligible. 
Now, we will estimate the first term in \eqref{eq:decomp}. 
For this purpose, let $d = (1-\gamma)/2$ and note that, for $n$ large enough,
\begin{align*}
\begin{split}
\label{eq:bedApriori}
\sum_{j=0}^{k-1} \lceil  nd\gamma^j  \rceil 
&\leq
\sum_{j=0}^{k-1}  nd\gamma^j  + k \\
&=
nd\frac{1-\gamma^k}{1-\gamma}  + k \\
&\leq
n/2  - \ln(n)/\ln(\gamma) + 1 \\
&\leq n. 
\end{split}
\end{align*}
Therefore, by inequality \eqref{eq:eppel} in the third step and inequality
\eqref{eq:assumptionModel} and Lemma~\ref{lem:assumptionModelConsequence} in the fourth step, we obtain
\begin{align*}
\P_x( \tau > n , T_k > n )
&\leq
\sum_{j=0}^{k-1} \P_x( \tau \geq T_j , T_{j+1}-T_{j} > \lceil nd\gamma^j  \rceil ) \\
&=
\sum_{j=0}^{k-1} \int p_y^{(T)}(\lceil nd\gamma^j  \rceil) \P_x(H_j \in \di y , \tau \geq T_j) \\
&\leq
\sum_{j=0}^{k-1} \int  \frac{c(|y|+1)}{\sqrt{nd\gamma^j }} \P_x(H_j \in \di y , \tau \geq T_j) \\
&\leq
\sum_{j=0}^{k-1}  \frac{c\gamma^j(|x|+1)}{\sqrt{nd\gamma^j }}  \\
&\leq
c (|x|+1) n^{-1/2},
\end{align*}
which completes the proof.
\end{proof}

The proofs of our main results are based on the next lemma.
It allows us to replace persistence events by easier to handle events where all zero-crossings occur at the beginning and overshoots over zero are relatively small.

\begin{lem}
\label{lem:jumpsBeginning}
Let $c>0$. Assume that \eqref{eq:assumptionModel} holds.
Then, we have uniformly in 
$\{ x \colon |x|\leq a_n n^{1/2}, \P_x( \tau > n ) \geq c^{-1}(|x|+1)n^{-1/2} \}$ that, as $n \to \infty$,
\begin{equation*}
\P_x( \tau > n )
\sim
\P_x( \exists k \geq 0 \colon |H_k| \leq a_n n^{1/2} , T_k \leq b_n ,  T_{k+1} > n , \tau > n ).
\end{equation*}
\end{lem}

\begin{proof}
The proof is organized into two parts. First, we will prove that 
\[
\sup_{|x| \leq a_n n^{1/2}} \frac{\P_x( \exists k \geq 0 \colon T_k \in [b_n,n], \tau > n )}{|x|+1}  = o(n^{-1/2}).
\] 
Then, we will show that 
\[
\sup_{|x| \leq a_n n^{1/2}} \frac{\P_x( \exists k \geq 0 \colon |H_k| > a_n n^{1/2}, T_k < b_n, T_{k+1} > n,  \tau > n )}{|x|+1}  = o(n^{-1/2}).\] 
Combining both statements, completes the proof.

\emph{First part:} 
By Lemma \ref{lem:apriori}, we obtain that 
\begin{align}
\label{eq:lem:jumpsBeginning:step1}
\begin{split}
\P_x( \exists k \geq 0 & \colon T_k \in [b_n , (1-\delta) n] , \tau > n )\\
&\leq
\sum_{k=0}^\infty \P_x( T_k \in [b_n , (1-\delta) n] , \tau > n ) \\
&\leq
\sum_{k=0}^\infty\int p_y^{(\tau)}(\lceil\delta n\rceil)  \P_x( H_k \in \di y , T_k > b_n , \tau \geq T_k ) \\
&\leq
\sum_{k=0}^\infty\int \frac{c(|y|+1)}{\sqrt{\delta n}} \P_x( H_k \in \di y , T_k > b_n , \tau \geq T_k ).
\end{split}
\end{align}
For our next argument, let us recall that, 
by (b1) and (b2) in Lemma \ref{lem:estimates_an_bn}, we have, for $|x| \leq a_n n^{1/2}$ and some sequence $c_n^{(k)}=o(1)$,
\[
\int (|y|+1) \P_x( H_k \in \di y , T_k > b_n ) \leq c_n^{(k)}(|x|+1),
\]
and that, 
by Lemma \ref{lem:assumptionModelConsequence} and \eqref{eq:assumptionModel}, we have, for all $x \in \R$,
\[
\int (|y|+1) \P_x( H_k \in \di y , \tau \geq T_k ) \leq c \gamma^k(|x|+1).
\]
Now, let $\varepsilon>0$. Let us choose $k_0$ such that $\sum_{k=k_0}^\infty c \gamma^k \leq \varepsilon/2$ 
and 
$n_0$ such that $ c_n^{(k)} \leq \varepsilon/(2 k_0)$ for $k < k_0$ and $n \geq n_0$. Then, we obtain
\begin{align*}
\sum_{k=0}^\infty\int & (|y|+1) \P_x( H_k \in \di y , T_k > b_n , \tau \geq T_k )\\
\leq&
\sum_{k=0}^{k_0-1} \int (|y|+1) \P_x( H_k \in \di y , T_k > b_n  ) 
\\&+ 
\sum_{k=k_0}^\infty\int (|y|+1) \P_x( H_k \in \di y , \tau \geq T_k )\\
\leq&
\varepsilon (|x|+1) .
\end{align*}
Since $\varepsilon$ was arbitrary, it follows that 
\begin{equation}
\label{eq:unif_b_n}
\sup_{|x| \leq a_n n^{1/2}} \frac{\sum_{k=0}^\infty\int (|y|+1) \P_x( H_k \in \di y , T_k > b_n , \tau \geq T_k )}{|x|+1}
=
o(1).
\end{equation}
By \eqref{eq:lem:jumpsBeginning:step1}, we thus obtain that 
\begin{equation}
\label{eq:crossings AfterBn1}
\sup_{|x| \leq a_n n^{1/2}} \frac{\P_x( \exists k \geq 0 \colon T_k \in [b_n , (1-\delta) n] , \tau > n )}{|x|+1} = o((\delta n)^{-1/2}).
\end{equation}

Now, we will consider the case of zero-crossings between $(1-\delta)n$ and $n$.
Due to inequality \eqref{eq:eppel0} in the second step and Lemma \ref{lem:assumptionModelConsequence} and  inequality \eqref{eq:assumptionModel} in the fourth step, we obtain
\begin{align*}
&
\P_x(\exists k \geq 0 \colon T_k < b_n , (1-\delta)n < T_{k+1} \leq n , \tau > n  ) \\
&\quad\leq
\sum_{k=0}^\infty \P_x( T_k < b_n  , (1-\delta) n - b_n \leq T_{k+1}-T_k \leq n , \tau \geq T_k  ) \\
&\quad\leq
\sum_{k=0}^\infty \int \sum_{m=\lfloor (1-\delta) n - b_n \rfloor }^n \frac{c(|y|+1)}{m^{3/2}} \P_x(  H_k \in \di y ,  \tau \geq T_k ) \\
&\quad\leq
\sum_{k=0}^\infty \int \frac{c(|y|+1)}{\sqrt{n}} \left(  \frac{\sqrt{n}}{\sqrt{\lfloor (1-\delta) n - b_n \rfloor -1}} -1 \right) \P_x(  H_k \in \di y ,  \tau \geq T_k ) \\
&\quad\leq
\sum_{k=0}^\infty  \frac{c \gamma^k (|x|+1)}{\sqrt{n}} \left( \frac{\sqrt{n}}{\sqrt{\lfloor (1-\delta) n - b_n \rfloor -1}} -1 \right)  \\
&\quad\leq
\frac{c(|x|+1)}{\sqrt{n}} \left( \frac{\sqrt{n}}{\sqrt{\lfloor (1-\delta) n - b_n \rfloor -1}} -1 \right)  .
\end{align*} 
Now, letting $\delta \to 0$, as $n \to \infty$, we can conclude that
\begin{equation}
\label{eq:crossings AfterBn2}
\sup_{x \in \R} \frac{\P_x(\exists k \geq 0 \colon T_k < b_n , (1-\delta)n < T_{k+1} \leq n , \tau > n  )}{|x|+1} = o(n^{-1/2}).
\end{equation} 
Further, if $\delta \to 0$ slowly enough, so that \eqref{eq:crossings AfterBn1} is in 
$o(n^{-1/2})$, we can combine \eqref{eq:crossings AfterBn1} and \eqref{eq:crossings AfterBn2} to obtain  that 
\begin{equation*}
\sup_{|x| \leq a_n n^{1/2}} \frac{\P_x( \exists k \geq 0 \colon T_k \in [b_n,n] , \tau > n )}{|x|+1}  = o(n^{-1/2}).
\end{equation*}

\emph{Second part:} 
It remains to control the probability of observing a surviving path with a large overshoot at the last zero-crossing time before $n$. 
By \eqref{eq:eppel}, we obtain
\begin{align}
\label{eq:lem:jumpsBeginning:step2}
\begin{split}
&
\P_x( \exists k \geq 0 \colon |H_k| > a_n n^{1/2} , T_k < b_n , T_{k+1}>n ,  \tau > n ) \\
&\leq
\sum_{k=0}^\infty \P_x( |H_k| > a_n n^{1/2} , T_k < b_n ,  T_{k+1}-T_k > n-b_n , \tau \geq T_k ) \\
&\leq
\sum_{k=0}^\infty p_y^{(T)}(n-b_n)  \P_x( H_k \in \di y  , |H_k| > a_n n^{1/2} , \tau \geq T_k ) \\
&\leq
\sum_{k=0}^\infty \frac{c(|y|+1)}{\sqrt{n-b_n}} \P_x( H_k \in \di y  , |H_k| > a_n n^{1/2} , \tau \geq T_k ) .
\end{split}
\end{align}
Now, along the same lines as in the proof of \eqref{eq:unif_b_n}, we obtain, using Lemma \ref{lem:assumptionModelConsequence}, inequality \eqref{eq:assumptionModel} and identities (a1), (a2) in Lemma \ref{lem:estimates_an_bn}, that 
\begin{equation}
\label{eq:unif_a_n}
\sup_{|x| \leq a_n n^{1/2}} \frac{\sum_{k=0}^\infty\int (|y|+1) \P_x( H_k \in \di y , |H_k|>a_n n^{1/2} , \tau \geq T_k )}{|x|+1} 
=
o(1).
\end{equation}
Thus, combing \eqref{eq:lem:jumpsBeginning:step2} and \eqref{eq:unif_a_n}, we obtain that 
\begin{equation*}
\sup_{|x| \leq a_n n^{1/2}} \frac{\P_x( \exists k \geq 0 \colon |H_k| > a_n n^{1/2} , T_k < b_n , T_{k+1}>n ,  \tau > n )}{|x|+1}  = o(n^{-1/2}).
\end{equation*}
\end{proof}

\section{Proofs of the main results}
\label{sec:results}

\subsection{\texorpdfstring{Proof of Theorem~\ref{thm:persistence}}{Proof of Theorem 1}}
\label{sub:proofThm1}

We will first prove the non-uniform case and will prove the uniform case afterwards. 
In both cases the proof is based on the following application of Lemma \ref{lem:jumpsBeginning}. 
We obtain that,
uniformly in 
$\{ x \colon |x| \leq a_n n^{1/2}, \P_x( \tau > n ) \geq c^{-1}(|x|+1)n^{-1/2}\}$,
\begin{align}
\label{eq:begProofThm}
\begin{split}
n^{1/2}& \cdot  \P_x( \tau > n )\\
&\sim
n^{1/2} \cdot \P_x( \exists k \geq 0 \colon |H_k| \leq a_n n^{1/2} , T_k \leq b_n ,  T_{k+1} > n , \tau > n ) \\
&=
\sum_{k=0}^\infty \int n^{1/2} \cdot p_y^{(T,\tau)}(n-t) \1{\{ |y|\leq a_nn^{1/2} , t\leq b_n \}} \\& \qquad \qquad \qquad \qquad \qquad \qquad \P_x( H_k \in \di y , T_k \in \di t , \tau \geq T_k ) \\
&\sim
\sum_{k=0}^\infty \int (n-t)^{1/2} \cdot p_y^{(T,\tau)}(n-t) \1{\{|y|\leq a_nn^{1/2}, t\leq b_n \}} \\& \qquad \qquad \qquad \qquad \qquad \qquad \P_x( H_k \in \di y , T_k \in \di t , \tau \geq T_k ),
\end{split}
\end{align}
where $(b_n)$ is the sequence defined in Subsection \ref{sub:notation}.
In particular, \eqref{eq:begProofThm} holds for fixed $x$ satisfying \eqref{eq:assumptionOrder0}.
By \eqref{eq:thmConvU}, we have, for all $y \in \R$,
\[
(n-t)^{1/2} \cdot p_y^{(T,\tau)}(n-t) \1{\{ t\leq b_n , |y|\leq a_n n^{1/2} \}} \to u(y).\]
Further, 
by \eqref{eq:eppel}, for all $y \in \R$,
\begin{align*}
n^{1/2} \cdot p_y^{(T,\tau)}(n)
\leq
n^{1/2} \cdot \P_y(T_1 > n) 
\leq
c(|y|+1) .
\end{align*}
Hence, the statement for fixed $x$ follows immediately from \eqref{eq:begProofThm} 
by Lemma \ref{lem:assumptionModelConsequence}, inequality \eqref{eq:assumptionModel} and  the dominated convergence theorem.

Now, we consider the uniform case and thus assume that 
\begin{equation}
\label{eq:unifCondProof0}
\sup_{|y| \leq a_n n^{1/2}} \left| n^{1/2} p_y^{(T,\tau)}(n) -  u(y) \right| = o(1).
\end{equation} 
Let $a_n':= a_{n-b_n}(1-b_n/n)^{1/2}$. Then, one easily obtains that $a_n'n^{1/2} \nearrow \infty$ and $(a_n')^2 n = o(b_n)$. Thus, since $a_n' n^{1/2} = a_{n-b_n}(n-b_n)^{1/2} \leq a_{n-t}(n-t)^{1/2}$, for $t \leq b_n$, one obtains from \eqref{eq:unifCondProof0} that
\begin{equation}
\label{eq:unifCondProof}
\sup_{|y| \leq a_n' n^{1/2}, t \leq b_n} \left| (n-t)^{1/2} p_y^{(T,\tau)}(n-t) -  u(y) \right| = o(1).
\end{equation}
Further, we recall that, due to the assumption $\P_x( \tau > n ) \geq c^{-1}(|x|+1)n^{-1/2}$, we have $n^{1/2} \cdot \P_x( \tau > n ) \geq c^{-1}(|x|+1)$.
Hence, we obtain that,
uniformly in 
$\{ x \colon |x| \leq a_n' n^{1/2}, \P_x( \tau > n ) \geq c^{-1}(|x|+1)n^{-1/2}\}$,
\begin{align*}
n^{1/2} \cdot \P_x( \tau > n )
&\sim
\sum_{k=0}^\infty \int u(y) \1{\{|y|\leq a_n' n^{1/2}, t\leq b_n  \}} \P_x( H_k \in \di y , T_k \in \di t , \tau \geq T_k ) \\
&\sim
\sum_{k=0}^\infty \int u(y) \P_x( H_k \in \di y  , \tau \geq T_k ) \\
&=
\sum_{k=0}^\infty \E_x[ u(H_k) ; \tau \geq T_k ] ,
\end{align*}
where we used \eqref{eq:begProofThm}, \eqref{eq:unifCondProof} and inequality \eqref{eq:assumptionModel} in the first step. 
In the second step, we used inequalities \eqref{eq:unif_b_n} and \eqref{eq:unif_a_n} combined with the fact that $u(y) \leq c(|y|+1)$. 
This finishes the proof.

\subsection{\texorpdfstring{Proof of Theorem~\ref{thm:scaling}}{Proof of Theorem 2}}
\label{sub:proofThm2}

The idea of the proof is based on the observations 
after the statement of Theorem~\ref{thm:scaling}: 
A random walk that survives until time $n$ typically crosses zero only a few times at the beginning and then stays on one side of zero. One would expect that this beginning part should disappear in the scaling limit.
Our main tools to show this are Lemma~\ref{lem:jumpsBeginning} from above and  Theorem~5.5 in \cite{Billingsley1968}, which is an extension of the continuous mapping theorem that allows us to replace the continuous function in the classical theorem by certain sequences of continuous functions. In our situation it states the following.
Let $(W_n)$ be a sequence of stochastic processes and $W$ be stochastic process in $(C([0,1]),\|\cdot\|_\infty)$. Let $\tilde{\Theta}_n \colon (C([0,1]),\|\cdot\|_\infty) \to (C([0,1]),\|\cdot\|_\infty)$ be measurable functions for $n=1,2,\ldots$ and \[\{ h \in C([0,1]) \colon \exists (h_n) \subseteq C([0,1]) \text{ s.t. } \| h_n - h \|_\infty \to 0 ,\ \| \tilde{\Theta}_n(h_n) - \tilde{\Theta}(h) \|_\infty \not\to 0 \}\]
be a null set w.r.t. $Law(W)$.
Then,
\begin{equation}
\label{eq:thm5.5}
Law(W_n) \Rightarrow Law(W) \quad \text{implies} \quad Law(\tilde{\Theta}_n(W_n)) \Rightarrow Law(\tilde{\Theta}(W)).
\end{equation}

We begin by considering a modified random walk, which is composed of a fixed beginning part and a normal random walk. We will show that the scaling limit of this process does not depend on the beginning part.
For this purpose, let us introduce some notation. 
Let $t_0>0$ and let $g \in C([0,t_0])$. 
We denote by $g^{(n)} \in C([0,t_0/n])$ the rescaled version of $g$, defined by $g^{(n)}(t):=g(t/n)/(\sigma n^{1/2})$. Further, for $n \geq t_0$ and $h \in C([0,1])$, we denote by $\Theta_n(g^{(n)},h) \in C([0,1])$ the function given by
\[
\Theta_n(g^{(n)},h)(t) = \begin{cases} g^{(n)}(t), &  t < t_0/n,\\ h(t-t_0/n)-h(0)+g^{(n)}(t_0/n), &   t_0/n \leq t \leq 1. \end{cases}
\]
Note that $\Theta_n(g^{(n)},\cdot)$ is continuous on $C([0,1])$ with respect to the supremum norm $\| \cdot \|_\infty$.
Now, we want to study weak limits of the rescaled processes $\Theta_n(g^{(n)},\hat{S}_n)$ in $(C([0,1]),\| \cdot \|_\infty)$. 
Let $f \colon C([0,1]) \to \R$ be a continuous and bounded function.
Then, we are interested in the limit 
of the quantity 
\begin{equation}
\label{eq:modifiedSequence}
d_n(f,g^{(n)})  := \frac{\E_{y}[ f(\Theta_n(g^{(n)} , \hat{S}_n )) ; T_1 > n, \tau > n ]}{\P_y( T_1 > n , \tau > n )}, \quad \text{as } n \to \infty, 
\end{equation}
where 
$y=g(t_0)$. 
Let us assume that $u(y)>0$. 
Then, by \eqref{eq:thmScalingAssumption}, we have 
\[
\frac{\E_{y}[ f( \hat{S}_n ) ; T_1 > n, \tau > n ]}{\P_y( T_1 > n , \tau > n )} \to \begin{cases} \E[f(W_+)] , &  y \geq 0,\\ \E[f(W_-)] , & y < 0. \end{cases}
\]
In the next step, we will replace this sequence by the sequence in \eqref{eq:modifiedSequence} and apply the earlier mentioned extension of the continuous mapping theorem \eqref{eq:thm5.5} with $\tilde{\Theta}_n(\cdot) := \Theta_n(g^{(n)},\cdot)$. 
Thus, we have
\begin{align}
\label{eq:proofThm2Conv}
d_n(f,g^{(n)}) \to 
\begin{cases} \E[f(W_+)] , &  y \geq 0,\\ \E[f(W_-)] , & y < 0, \end{cases}
\end{align}
if $\| \Theta_n(g^{(n)},h_n) - h \|_\infty \to 0$ for all $h,h_1,h_2,\ldots \in C([0,1])$ with $h(0)=0$ and $\lim_{n \to \infty} \|h-h_n\|_\infty = 0$.
%
But this follows already from the following simple calculation. We have 
\begin{align*}
\| \Theta_n(g^{(n)},h_n) - h \|_\infty
\leq&
\| \Theta_n(g^{(n)},h_n) - \Theta_n(g^{(n)},h) \|_\infty \\ &+ \| \Theta_n(g^{(n)},h) - \Theta_n(0,h) \|_\infty + \| \Theta_n(0,h) - h \|_\infty \\
\leq&
(\| h_n - h \|_\infty + |h_n(0)|) 
+ \| g^{(n)} \|_\infty 
+ \| \Theta_n(0,h) - h \|_\infty\\
\to&
0.
\end{align*}
Clearly, the first term and the second term tend to $0$. 
Since $h$ is uniformly continuous, as a fixed continuous function on a compact interval, also the last term tends to $0$.

In the next step, we provide a slightly modified version of Lemma \ref{lem:jumpsBeginning}. 
Due to \eqref{eq:thmConvU} and the dominated convergence theorem, we can replace the integrand in \eqref{eq:begProofThm} by $n^{1/2} \cdot p_y^{(T,\tau)}(n)$ and obtain
\begin{equation*}
\P_x( \tau > n )
\sim
\P_x( \exists k \geq 0 \colon T_k \leq b_n , T_{k+1}-T_k>n,  \tau > n+T_k ).
\end{equation*}
Therefore, since 
\begin{equation}
\label{eq:lem6modEvents}
\{ \exists k \geq 0 \colon T_k \leq b_n , T_{k+1}-T_k>n,  \tau > n+T_k \} \subseteq \{ \tau > n\},
\end{equation}
and $f$ is bounded, we obtain 
\begin{align*}
& \left| \frac{\E[f(\hat{S}_n);\tau>n]}{\P(\tau>n)} - \frac{\E[f(\hat{S}_n);\exists k \geq 0 \colon T_k \leq b_n ,  T_{k+1}-T_k>n,  \tau > n+T_k]}{\P(\tau>n)}  \right| \\
& \leq
\| f \|_\infty \cdot \frac{\P( \{ \tau > n \} \setminus \{  \exists k \geq 0 \colon T_k \leq b_n ,  T_{k+1}-T_k>n,  \tau > n+T_k \}  )}{\P(\tau>n)} \to 0 .
\end{align*}
Thus, in the following, we can replace the event $\{ \tau > n \}$ by the easier to handle left event in \eqref{eq:lem6modEvents}.

We recall that, by \eqref{eq:eppel}, 
\begin{equation}
\label{eq:proofThm2Bound}
n^{1/2} \cdot p_y^{(T,\tau)}(n) | d_n(f,g^{(n)}) | \leq c(|y|+1) \| f \|_\infty.
\end{equation}
Further, we set $\hat{S}_n^{(t)} := \hat{S}_n |_{[0,t/n]} \in C([0,t/n])$.
Then, by Theorem \ref{thm:persistence} in the second step and by \eqref{eq:thmConvU}, \eqref{eq:proofThm2Conv}, \eqref{eq:proofThm2Bound} and dominated convergence in the third step, we obtain
\begin{align*}
&\frac{1}{p_x^{(\tau)}(n)} \E_x[ f(\hat{S}_n) ; \exists k \geq 0 \colon T_k \leq b_n ,  T_{k+1}-T_k>n,  \tau > n+T_k ]\\
&=
\sum_{k=0}^\infty \int \frac{p_y^{(T,\tau)}(n)}{p_x^{(\tau)}(n)} d_n(f,g^{(n)}) \1{\{ t \leq b_n  \}} \\ & \qquad \qquad \qquad \qquad \qquad \qquad \P_x( H_k \in \di y , T_k \in \di t, \hat{S}_n^{(t)} \in \di g^{(n)}, \tau \geq T_k )  \\
&\sim
V(x)^{-1}\sum_{k=0}^\infty \int n^{1/2} \cdot p_y^{(T,\tau)}(n) d_n(f,g^{(n)}) \1{\{ t \leq b_n  \}} \\ & \qquad \qquad \qquad \qquad \qquad \qquad \P_x( H_k \in \di y , T_k \in \di t, \hat{S}_n^{(t)} \in \di g^{(n)} , \tau \geq T_k )  \\
&\sim
V(x)^{-1} \sum_{k=0}^\infty \int u(y)   \left( \E[f(W_+)]\1{\{y \geq 0\}} + \E[f(W_-)]\1{\{y < 0\}} \right) \\ & \qquad \qquad \qquad \qquad \qquad \qquad \qquad \qquad \qquad \qquad  \P_x( H_k \in \di y , \tau \geq T_k )  \\
&=
\P(\rho=1) \E[f(W_+)] + \P(\rho=0) \E[f(W_-)].
\end{align*}

\section*{Acknowledgements} 
I am very grateful to V. Vysotsky for drawing my attention to 
the work \cite{Vysotsky2015} and to F. Aurzada for the valuable comments which helped improve the presentation and the clarity of the paper.


\end{document}